\documentclass[english]{extarticle}
\usepackage[T1]{fontenc}
\usepackage[latin9]{inputenc}
\usepackage{geometry}
\usepackage{amsmath}
\usepackage{amsthm}
\usepackage{amssymb}
\usepackage{mathdots}
\usepackage{setspace}
\makeatletter
\numberwithin{equation}{section}
\numberwithin{figure}{section}
\theoremstyle{plain}
\newtheorem{thm}{\protect\theoremname}
  \theoremstyle{remark}
  \newtheorem{claim}[thm]{\protect\claimname}
  \theoremstyle{plain}
  \newtheorem{cor}[thm]{\protect\corollaryname}
  \theoremstyle{plain}
  \newtheorem{prop}[thm]{\protect\propositionname}
  \theoremstyle{remark}
  \newtheorem{rem}[thm]{\protect\remarkname}
  \theoremstyle{plain}
  \newtheorem{lem}[thm]{\protect\lemmaname}

\usepackage[nottoc]{tocbibind}

\makeatother

\usepackage{babel}
  \providecommand{\claimname}{Claim}
  \providecommand{\corollaryname}{Corollary}
  \providecommand{\lemmaname}{Lemma}
  \providecommand{\propositionname}{Proposition}
  \providecommand{\remarkname}{Remark}
\providecommand{\theoremname}{Theorem}

\begin{document}

\title{UPPER BOUND ON THE MINIMAL NUMBER OF RAMIFIED PRIMES FOR ODD ORDER
SOLVABLE GROUPS }

\author{RABAYEV DANIEL}
\maketitle
\begin{abstract}
{\normalsize{}Let $G$ be a finite group and let $ram^{t}(G)$ denote
the minimal positive integer $n$ such that $G$ can be realized as
the Galois group of a tamely ramified extension of $\mathbb{Q}$ ramified
only at $n$ finite primes. Let $d(G)$ denote the minimal non negative
integer for which there exists a subset $X$ of $G$ with $d(G)$
elements such that the normal subgroup of $G$ generated by $X$ is
all of $G$. It is known that $d(G)\leq ram^{t}(G)$. However, it
is unknown whether or not every finite group $G$ can be realized
as a Galois group of a tamely ramified extension of $\mathbb{Q}$
with exactly $d(G)$ ramified primes. We will show that $3\cdot log(|G|)$
is an upper bound for $ram^{t}(G)$ for all odd order solvable group
$G$.}{\normalsize \par}

\newcommand\blfootnote[1]{%
\begingroup   
\renewcommand\thefootnote{}\footnote{#1}%
\addtocounter{footnote}{-1}%
\endgroup 
}

\end{abstract}

\section{{\large{}Introduction}}

The inverse Galois problem states that for every group $G$, there
should exist a Galois extension $K/\mathbb{Q}$ with a Galois group
which is isomorphic to $G$. It is evident that odd order groups are
solvable groups, and we know \cite{key-12} that every finite odd
(solvable) group is realizable over $\mathbb{Q}$. A variant of the
inverse Galois problem is the minimal ramification problem, which
we will now discuss.

Let $ram^{t}(G)$ denote the minimal positive integer $n$ such that
$G$ can be realized as the Galois group of a tamely ramified extension
of $\mathbb{Q}$ ramified only at $n$ finite primes, and let $d(G)$
denote the minimal nonnegative integer for which there exists a subset
$X$ of $G$ with $d(G)$ elements such that the normal subgroup of
$G$ generated by $X$ is all of $G$. 

Let $K/\mathbb{Q}$ be a finite Galois extension with Galois group
$G=\textrm{Gal}(K/\mathbb{Q})$, where $G$ is a finite group. Let
$p$ be a finite prime of $\mathbb{Q}$. If $p$ ramifies in $K$
and if $\mathfrak{p}$ is a prime of $K$ dividing $p$, then the
inertia group $I(\mathfrak{p}\mid p)$ is a nontrivial subgroup of
$G$. If $I$ is the subgroup of $G$ generated by all $I(\mathfrak{p}\mid p)$,
then the fixed field of $I$ is an unramified extension of $\mathbb{Q}$.
Since by Minkowski\textquoteright s theorem, there are no nontrivial
unramified extensions of $\mathbb{Q}$, we must have that $I=G$.
Suppose in addition that $K/\mathbb{Q}$ is tamely ramified, i.e.
for every prime $p$ which ramifies in $K$, all of the inertia groups
are cyclic and of order prime to $p$. In this case, let us denote
$I(\mathfrak{p}\mid p)=<g_{p}>$ and we deduce that the normal subgroup
of $G$ generated by all of the $g_{p}$ is $G$. We conclude that
$d(G)\leq ram^{t}(G)$. 

Whether or not every finite group can be realized as a Galois group
of a tamely ramified extension of $\mathbb{Q}$ with exactly $d(G)$
ramified primes is an open question. The (tame) minimal ramification
problem is the following: can every finite group $G$ be realized
as a Galois group of a tamely ramified extension of $\mathbb{Q}$
with exactly $d(G)$ ramified primes. In \cite{key-8}, Boston and
Markin conjectured that the minimal ramification problem has a positive
answer, namely that for every finite group $G$ we have $d(G)=ram^{t}(G)$.
Most of the known results are valid only for $l$-groups, where $l$
is an odd prime. Moreover, most of these results give an upper bound
on $ram^{t}(G)$. An important property of the bound is that one should
be able to calculate it directly from the structure of $G$ as an
abstract group, namely the bound should not depend on the realization
of the group as a Galois group. Serre \cite{key-4} noted, using the
Scholz-Reichardt method, that for a finite $l$-group $G$, $|G|=l^{n}$,
$l$ being an odd prime: $ram^{t}(G)\leq n$. However, $n=d(G)$ only
if $G$ is elementary abelian. Geyer and Jarden \cite{key-6} generalized
Serre's result for global fields. Namely, they showed that for a global
field $K$ with $l\neq char(K)$ and $\mu_{l}\nsubseteq K$ (note
that $l\neq2$), there exists $r=r(K)$ such that for every $l$-group
$G$ of order $l^{n}$: $ram^{t}(G)\leq n+r$. In particular, for
$K=\mathbb{Q}$ we have $r=0$.\textbf{ }Plans \cite{key-5} sharpened
Serre's upper bound over $\mathbb{Q}$ by showing that the Scholz-Reichardt
method yields the following bound: $d(G)+\sum_{1\leq i\leq n-2}d(G_{i})$,
where $G_{i}=C_{i}/C_{i+1}$, $C_{i}$ being the descending central
series of $G$, and for $n\leq2$ the sum equals zero. We deduce that
the minimal ramification problem has a positive answer for all $l$-groups
$G$ of nilpotency class 2, where $l$ is an odd prime. Kisilevsky,
Neftin and Sonn \cite{key-9} proved that the minimal ramification
problem has a positive answer for the family of semiabelian\footnote{See \cite{key-17}.}
nilpotent groups. However, this family does not contain all finite
nilpotent groups, for example it is shown in \cite{key-2} that there
are 10 groups of order $64$ which are not semiabelian. 

We will give an upper bound on the number of ramified primes for all
finite odd order groups. In particular, we will prove the following
theorem:

\textbf{\textit{Theorem \ref{thm:main.neukirch}.}}\textit{ Let $G$
be an odd order group, then:}
\begin{eqnarray*}
 & ram^{t}(G)\leq3\cdot ln(|G|)\\
\end{eqnarray*}

This upper bound is extracted from Neukirch's proof of the realization
of odd order groups over number fields. Similar to previous results,
the upper bound is in a form of a sum of ranks of a certain derived
series (in this case it is the chief series) of the group. 

\section{{\large{}Cohomology and Ramification in number fields.}}

Let $K$ be a number field and denote $G_{K}$ the absolute Galois
group of $K$ and let $A$ be a $G_{K}$-module. As usual, $\textrm{H}^{q}(K,A)$
is the cohomology group $\textrm{H}^{q}(G_{K},A)$. Let $K_{nr}$
be the maximal unramified extension of $K$. 

We recall that the unramified cohomology group is defined to be the
image of the following: 
\begin{eqnarray*}
 & \textrm{H}^{q}(\textrm{Gal}(K_{\mathfrak{p},nr}/K_{\mathfrak{p}}),A)\overset{inf}{\longrightarrow}\textrm{H}^{q}(K_{\mathfrak{p}},A)
\end{eqnarray*}
where $\mathfrak{p}$ is a prime of $K$. We denote $\textrm{H}_{nr}^{q}(K_{\mathfrak{p}},A)=im(inf)$.
Consider the homomorphism:
\begin{eqnarray*}
 & \textrm{H}^{q}(K,A)\rightarrow\underset{\mathfrak{p}}{\prod}^{'}\textrm{H}^{q}(K_{\mathfrak{p}},A)
\end{eqnarray*}
 where the restricted product is taken with respect to $\textrm{H}_{nr}^{q}(K_{\mathfrak{p}},A)$.
We say that $x\in\textrm{H}^{q}(K,A)$ is unramified at $\mathfrak{p}$
if $x_{\mathfrak{p}}\in\textrm{H}_{nr}^{q}(K_{\mathfrak{p}},A)$,
otherwise we say that $x_{\mathfrak{p}}$ is ramified.

Let $p$ be an odd prime number and assume that char($K$)$\neq p$.
Let $\mu_{p}$ be the group of $p$-th roots of unity, assume that
$\mu_{p}\nsubseteq K$, denote $\mathfrak{K}=K(\mu_{p})$ and let
$\Delta=\textrm{Gal}(\mathfrak{K}/K).$ Let $A$ be a finite $\Delta$-module
and a trivial $G_{\mathfrak{K}}$ module. An action of $\delta\in\Delta$
on the group $\textrm{H}^{1}(\mathfrak{K},A)$ is then defined, namely:
\[
(\delta\cdot\varphi)(\tau)=\delta(\varphi(\tau^{\delta}))=\delta(\varphi(\hat{\delta}^{-1}\tau\hat{\delta}))
\]
where $\tau\in G_{\mathfrak{K}},\ \varphi\in\textrm{H}^{1}(\mathfrak{K},A)$,
and $\hat{\delta}\in G_{K}$ is a lifting of $\delta$. A canonical
character $\theta:\Delta\rightarrow(\mathbb{Z}/p)^{*}$ is defined
by $\delta(\zeta)=\zeta^{\theta(\delta)}$, where $\delta\in\Delta$
and $\zeta\in\mu_{p}$. However, by abuse of notation we will write
the action additively, namely, $\delta(\zeta)=\theta(\delta)\zeta$.
Define the following elements in $\mathbb{Z}/p[\Delta]$ by:
\begin{eqnarray*}
 & e_{i}=\frac{1}{|\Delta|}\sum_{\delta\in\Delta}\theta(\delta)^{-i}\delta
\end{eqnarray*}
where $i\in\mathbb{Z}$. A simple calculation shows that $e_{i}$
are idempotent elements. The action of $e_{1}$ on $\varphi\in\textrm{H}^{1}(\mathfrak{K},\mu_{p})$
is the following:
\begin{eqnarray*}
 & e_{1}\cdot(\varphi)(\tau)=\frac{1}{|\Delta|}\sum_{\delta\in\Delta}\theta(\delta)^{-1}(\delta\cdot\varphi)(\tau)=\frac{1}{|\Delta|}\sum_{\delta\in\Delta}\theta(\delta)^{-1}\delta(\varphi(\hat{\delta}^{-1}\tau\hat{\delta}))=\\
 & =\frac{1}{|\Delta|}\sum_{\delta\in\Delta}\theta(\delta)^{-1}\theta(\delta)(\varphi(\hat{\delta}^{-1}\tau\hat{\delta}))=\frac{1}{|\Delta|}\sum_{\delta\in\Delta}\varphi(\hat{\delta}^{-1}\tau\hat{\delta})
\end{eqnarray*}

A simple calculation shows that $\gamma\in\Delta$ acts on $e_{1}\cdot\varphi\in e_{1}\cdot\textrm{H}^{1}(\mathfrak{K},\mu_{p})$
as multiplication by $\theta(\gamma)$. A similar calculation yields
that $\gamma\in\Delta$ acts trivially on $e_{0}\cdot\textrm{H}^{1}(\mathfrak{K},\mathbb{Z}/p)$. 
\begin{claim}
\label{e0h1->e1h1}The following:
\begin{eqnarray*}
\psi^{*}:e_{0}\cdot\textrm{H}^{1}(\mathfrak{K},\mathbb{Z}/p)\otimes\mu_{p} & \rightarrow & e_{1}\cdot\textrm{H}^{1}(\mathfrak{K},\mu_{p})\\
e_{0}\varphi\otimes\zeta & \mapsto & e_{1}(\psi\varphi)
\end{eqnarray*}
is a $\Delta$-module isomorphism, where $\psi:\mathbb{Z}/p\rightarrow\mu_{p}$
is a fixed isomorphism.
\end{claim}
\begin{proof}
Let us first notice the following: 
\begin{eqnarray*}
e_{1}(\psi\varphi) & = & \frac{1}{|\Delta|}\sum_{\delta\in\Delta}\psi\varphi(\hat{\delta}^{-1}\tau\hat{\delta})=\psi\left[\frac{1}{|\Delta|}\sum_{\delta\in\Delta}\varphi(\hat{\delta}^{-1}\tau\hat{\delta})\right]=\psi(e_{0}\varphi)\\
\end{eqnarray*}
where $\varphi\in\textrm{H}^{1}(\mathfrak{K},\mathbb{Z}/p)$. Now,
assume that $e_{0}\varphi_{1}\otimes\zeta=e_{0}\varphi_{2}\otimes\zeta$.
Thus $e_{0}\varphi_{1}=e_{0}\varphi_{2}$, and we have:
\begin{eqnarray*}
 & \psi^{*}(e_{0}\varphi_{1}\otimes\zeta)=e_{1}(\psi\varphi_{1})=\psi(e_{0}\varphi_{1})=\psi(e_{0}\varphi_{2})=e_{1}(\psi\varphi_{2})=\psi^{*}(e_{0}\varphi_{2}\otimes\zeta)\\
\end{eqnarray*}
Hence, $\psi^{*}$ is well defined. $\psi^{*}$ is clearly an homomorphism.
Let $e_{0}\varphi\otimes\zeta\in ker(\psi^{*})$, then:
\begin{eqnarray*}
 & 0=\psi^{*}(e_{0}\varphi\otimes\zeta)=e_{1}(\psi\varphi)=\psi(e_{0}\varphi)\\
\end{eqnarray*}
and the fact that $\psi$ is an isomorphism, implies that $e_{0}\varphi=0$,
and thus $ker(\psi^{*})=0$. The map $\psi^{*}$ is also onto since
for $e\eta\in e_{1}\cdot\textrm{H}^{1}(\mathfrak{K},\mu_{p})$, $\psi^{*}(\psi^{-1}\eta)=e_{1}\eta$.
Last, we need to show that $\psi^{*}$ respect the action of $\Delta$,
indeed:
\begin{eqnarray*}
 & \delta(\psi^{*}(e_{0}\varphi\otimes\zeta))=\delta(e_{1}(\psi\varphi))=\theta(\delta)e_{1}(\psi\varphi)=\theta(\delta)\psi^{*}(e_{0}\varphi\otimes\zeta)=\\
 & =\psi^{*}(e_{0}\varphi\otimes\theta(\delta)\zeta)=\psi^{*}(\delta(e_{0}\varphi)\otimes\delta(\zeta)))=\psi^{*}(\delta(e_{0}\varphi\otimes\zeta))
\end{eqnarray*}
\end{proof}
\begin{cor}
\label{cor kummer}$e_{1}\cdot\textrm{H}^{1}(\mathfrak{K},\mu_{p})$
is isomorphic as a $\Delta-$module to $\textrm{H}^{1}(\mathfrak{K},\mathbb{Z}/p)^{\Delta}\otimes\mu_{p}$.
\end{cor}
\begin{proof}
We already showed that every $\delta\in\Delta$ acts trivially on
$e_{0}\cdot\textrm{H}^{1}(\mathfrak{K},\mathbb{Z}/p)$, thus It is
clear that $e_{0}\cdot\textrm{H}^{1}(\mathfrak{K},\mathbb{Z}/p)$
is contained $\textrm{H}^{1}(\mathfrak{K},\mathbb{Z}/p)^{\Delta}$.
In general, for a $\mathbb{Z}/p[\Delta]-$module $A$, $e_{i}A$ is
the maximal submodule of $A$ on which the elements $\delta\in\Delta$
act as multiplication by $\theta(\delta)^{i}$. Hence we have equality.
Combining this with the previous claim and we obtain the desired result.
\end{proof}
By Kummer Theory:
\begin{eqnarray*}
\mathfrak{K}^{*}/(\mathfrak{K}^{*})^{p} & \longrightarrow & \textrm{H}^{1}(\mathfrak{K},\mu_{p})\\
\alpha & \mapsto & \varphi_{\alpha}
\end{eqnarray*}
 where $\varphi_{\alpha}(\sigma)=\sigma(\sqrt[p]{\alpha})/\sqrt[p]{\alpha}$,
is an isomorphism. This yields an isomorphism from $e_{1}\cdot\mathfrak{K}^{*}/(\mathfrak{K}^{*})^{p}$
to $e_{1}\cdot\textrm{H}^{1}(\mathfrak{K},\mu_{p})$.

From the five term exact sequence we get:
\begin{eqnarray*}
 & 0\longrightarrow\textrm{H}^{1}(\Delta,\mathbb{Z}/p)\longrightarrow\textrm{H}^{1}(K,\mathbb{Z}/p)\overset{res}{\longrightarrow}\textrm{H}^{1}(\mathfrak{K},\mathbb{Z}/p)^{\Delta}\longrightarrow\textrm{H}^{2}(\Delta,\mathbb{Z}/p)\longrightarrow\textrm{H}^{2}(K,\mathbb{Z}/p)
\end{eqnarray*}
note that $\textrm{H}^{1}(\Delta,\mathbb{Z}/p)=\textrm{H}^{2}(\Delta,\mathbb{Z}/p)=0$
since the order of $\Delta$ divides $p-1$, which is prime to $p$.
Thus we obtain that the restriction map is an isomorphism.

Let us summarize our results in the following claim.
\begin{prop}
\label{iso-mod}The following $\Delta-$modules:
\begin{eqnarray*}
 & \textrm{H}^{1}(K,\mathbb{Z}/p)\otimes\mu_{p}\cong\textrm{H}^{1}(\mathfrak{K},\mathbb{Z}/p)^{\Delta}\otimes\mu_{p}\cong e_{1}\cdot\textrm{H}^{1}(\mathfrak{K},\mu_{p})\cong e_{1}\cdot\mathfrak{K}^{*}/(\mathfrak{K}^{*})^{p}\\
\end{eqnarray*}
are isomorphic.
\end{prop}
\begin{rem}
This claim is used without proof in \cite[Main Lemma]{key-12} where
the field theoretical meaning is explained. 
\end{rem}
Let us denote the ideal in $O_{\mathfrak{K}}$ associated to an element
$\alpha\in\mathfrak{K}^{*}/(\mathfrak{K}^{*})^{p}$ as:
\begin{eqnarray*}
 & (\alpha)=\mathfrak{A}^{p}\cdot\mathfrak{P}_{1}^{a_{1}}\cdot\cdot\cdot\mathfrak{P}_{j}^{a_{j}}\\
\end{eqnarray*}
where $\mathfrak{A}$ is an ideal of $O_{\mathfrak{K}}$, $\mathfrak{P}_{i}$
is a prime ideal of $O_{\mathfrak{K}}$, and for all $1\leq i\leq j$
we have $1\leq a_{i}\leq p-1$. A prime ideal $\mathfrak{q}$ of $K$
is called $relatively\ prime$ to $\alpha\in\mathfrak{K}^{*}/(\mathfrak{K}^{*})^{p}$
if for every $1\leq i\leq j$ we have $\mathfrak{P}_{i}\nmid\mathfrak{q}.$
\begin{lem}
\label{lem:a->x q unram}Let $\mathfrak{q}$ be a prime of $K$. Assume
that $\mathfrak{q}$$\nmid p$ and that $\mathfrak{q}$ is relatively
prime to $\alpha\in\mathfrak{K}^{*}/(\mathfrak{K}^{*})^{p}$. If $x\otimes\zeta\in\textrm{H}^{1}(K,\mathbb{Z}/p)\otimes\mu_{p}$
is the element corresponding to $\alpha$ from proposition \ref{iso-mod}
then $x$ is not ramified at $\mathfrak{q}$.
\end{lem}
\begin{proof}
The element $\alpha$ corresponds by Kummer theory to an extension
$\mathfrak{K}(\sqrt[p]{\alpha})/\mathfrak{K}$ using a character in
$\textrm{H}^{1}(\mathfrak{K},\mu_{p})$ (as shown explicitly by the
isomorphism $\mathfrak{K}^{*}/(\mathfrak{K}^{*})^{p}\longrightarrow\textrm{H}^{1}(\mathfrak{K},\mu_{p})$).
Proposition \ref{iso-mod} shows us that this character actually comes
from a character $x\in\textrm{H}^{1}(K,\mathbb{Z}/p)$. Let $K_{x}$
denote the cyclic extension of $K$ defined by $x$. This means that
the compositum of $K_{x}$ with $\mathfrak{K}$ equals $\mathfrak{K}(\sqrt[p]{\alpha})$,
namely:

\begin{align*}
 & \begin{array}{ccc}
\mathfrak{K}\cdot K_{x} & = & \mathfrak{K}\left(\sqrt[p]{\alpha}\right)\\
\mid &  & \mid\\
K_{x} &  & \mathfrak{K}\\
\ \ \ \diagdown &  & \diagup\ \ \ \\
 & K
\end{array}\\
\end{align*}

Let $\mathfrak{q}$ be a prime of $K$ that is relatively prime to
$\alpha$ and satisfies that $\mathfrak{q}$$\nmid p$ . It is then
evident that $\mathfrak{q}$ is not ramified in $\mathfrak{K}(\sqrt[p]{\alpha})$.
Since $\mathfrak{K}(\sqrt[p]{\alpha})=\mathfrak{K}\cdot K_{x}$ we
deduce that $\mathfrak{q}$ is also unramified in $K_{x}$. It is
then clear that $(K_{x})_{\mathfrak{q}}$ is contained in $K_{\mathfrak{q},nr}$
and thus $x_{\mathfrak{q}}\in\textrm{H}_{nr}^{1}(K_{\mathfrak{q}},\mathbb{Z}/p)$,
namely, $x$ is unramified at $\mathfrak{q}$.
\end{proof}
Let $K\mid k$ be a Galois extension, $\Omega\mid K$ an abelian extension,
and assume that $\mu_{p}\nsubseteq K$. Let $A$ be a trivial $G_{K}$-module
with $pA=0$, $S$ be a finite set of primes of $K$, and let $y_{\mathfrak{p}}\in\textrm{H}^{1}(K_{\mathfrak{p}},\mathbb{Z}/p)$
for $\mathfrak{p}\in S$. In \cite[Main Lemma]{key-12} we obtain
an element $x\in\textrm{H}^{1}(K,A)$ such that:
\begin{enumerate}
\item $x_{\mathfrak{p}}=y_{\mathfrak{p}}$ for $\mathfrak{p}\in S$.
\item if $\mathfrak{p}\notin S$ then $x_{\mathfrak{p}}$ is cyclic, and
if $x_{\mathfrak{p}}$ is ramified then the prime $\mathfrak{p}_{0}=\mathfrak{p}\cap k$
of $k$, splits completely in $\Omega$ and $x_{\mathfrak{p}'}=0$
for all primes $\mathfrak{p}'\mid\mathfrak{p}_{0}$ of $K$ different
from $\mathfrak{p}$.
\end{enumerate}
This element $x\in\textrm{H}^{1}(K,A)$ is obtained as the corresponding
element to a certain $\alpha\in e_{1}\cdot\mathfrak{K}^{*}/(\mathfrak{K}^{*})^{p}$.
This element is calculated in \cite[Appendix]{key-12} and is the
product of two elements $\gamma\cdot\delta$. We are interested in
the properties of these elements outside of $S$. Indeed, it is shown
that they satisfy the following conditions:
\begin{enumerate}
\item $\delta_{\mathfrak{p}}\in U_{\mathfrak{p}}\mathfrak{K}_{\mathfrak{p}}^{*}/(\mathfrak{K}_{\mathfrak{p}}^{*})^{p}$
for $\mathfrak{p}\notin S\cup\{\mathfrak{q}\}$, where $\mathfrak{q}$
is a prime of $\mathfrak{K}$ which is not in $S$.
\item $\gamma_{\mathfrak{p}}\in U_{\mathfrak{p}}\mathfrak{K}_{\mathfrak{p}}^{*}/(\mathfrak{K}_{\mathfrak{p}}^{*})^{p}$
for $\mathfrak{p}\notin\{\mathfrak{p}_{i_{1}},\mathfrak{p}_{i_{2}}\}$,
where $\mathfrak{p}_{i_{1}},\mathfrak{p}_{i_{2}}$ are primes of $\mathfrak{K}$
which are not in $S$.
\end{enumerate}
We are interested in understanding the element $x\in\textrm{H}^{1}(K,A)$
which is obtained from the Main Lemma and corresponds to $\alpha\in\mathfrak{K}^{*}/(\mathfrak{K}^{*})^{p}$.
Let us first assume that $A=\mathbb{Z}/p$, and let $S$ be a finite
set of primes of $K$. Under the previous notation, for primes $\mathfrak{p}\nmid p$
and $\mathfrak{p}\notin S\cup\{\mathfrak{q},\mathfrak{p}_{i_{1}},\mathfrak{p}_{i_{2}}\}$
we know that $\alpha_{\mathfrak{p}}=\delta_{\mathfrak{p}}\gamma_{\mathfrak{p}}$
is a unit in $\mathfrak{K}_{\mathfrak{p}}^{*}/(\mathfrak{K}_{\mathfrak{p}}^{*})^{p}$
and thus relatively prime to $\alpha$. Hence, from lemma \ref{lem:a->x q unram}
we deduce that $x$ is not ramified at $\mathfrak{p}$. We deduce
that $x$ can be ramified outside of $S$ at no more than 3 primes.
The general case is done by induction over $dim_{\mathbb{Z}/p}A$.
For $x\in\textrm{H}^{1}(K,A)$ we have the following decomposition:
$x=x^{'}+x^{''}$, where $x^{'}\in\textrm{H}^{1}(K,A')$, $x^{''}\in\textrm{H}^{1}(K,\mathbb{Z}/p)$,
and $A=A'\oplus\mathbb{Z}/p$. Similarly, we have $x_{\mathfrak{p}}=x_{\mathfrak{p}}^{'}+x_{\mathfrak{p}}^{''}$
for $x_{\mathfrak{p}}\in\textrm{H}^{1}(K_{\mathfrak{p}},A)$. In the
proof of the \cite[Main Lemma]{key-12} it is shown how to find such
$x_{\mathfrak{p}}^{'},x_{\mathfrak{p}}^{''}$ with properties which
we will soon describe. Let $V$ denote a finite set of primes of $K$
which is closed under conjugation over $k$ and contains the set $S\cup\{\mathfrak{p}\mid\ x_{\mathfrak{p}}^{'}\ is\ ramified\}$.
Then we have:
\begin{enumerate}
\item $x_{\mathfrak{p}}=x_{\mathfrak{p}}^{'}+x_{\mathfrak{p}}^{''}=x_{\mathfrak{p}}^{'}$
for $\mathfrak{p}\in V-S$.
\item $x_{\mathfrak{p}}=x_{\mathfrak{p}}^{'}+x_{\mathfrak{p}}^{''}=x_{\mathfrak{p}}^{''}$
for $\mathfrak{p}\notin V$.
\end{enumerate}
It is then clear that $x$ is ramified outside of $S$ at the same
places $x^{'}$ and $x^{''}$ are ramified. $x^{''}$ is obtained
in the same way as the special case where $A=\mathbb{Z}/p$, namely
at most 3 primes outside of $S$. $x^{'}$ is obtained by induction
and it is evident that it is ramified outside of $S$ at no more then
$3\cdot rank(A')$ primes. To conclude, $x\in\textrm{H}^{1}(K,A)$
is ramified outside of $S$ at no more then $3\cdot rank(A)$ primes.
Let us conclude what we have just shown in the following proposition:
\begin{prop}
\label{prop:main.lemma}Under the previous notations, $x\in\textrm{H}^{1}(K,A)$
is ramified outside of $S$ at no more then $3\cdot rank(A)$ primes.
\end{prop}
We will need to work with elements of $\textrm{H}^{1}(k,A)$ rather
than elements of $\textrm{H}^{1}(K,A)$ and for this we have the following
theorem.
\begin{thm}
\label{theorem1.neu}\cite{key-12} Let $A$ be a simple $G_{k}$-module
with $pA=0$. Let $K\mid k$ be a Galois extension such that\footnote{For a finite $G_{k}$-module $A$ we denote by $k(A)\mid k$ the smallest
extension of $k$ for which $A$ is a trivial $G_{k(A)}$-module.} $k(A)\subseteq K$ but $\mu_{p}\nsubseteq K$ and let $\Omega\mid K$
be an abelian extension. Let $S$ be a finite set of primes of $k$
and $y_{\mathfrak{p}}\in\textrm{H}^{1}(k_{\mathfrak{p}},A)$ for $\mathfrak{p}\in S$.
Then there exists an element $x\in\textrm{H}^{1}(k,A)$ such that\\
1) $x_{\mathfrak{p}}=y_{\mathfrak{p}}$ for $\mathfrak{p}\in S$.\\
2) If $\mathfrak{p}\notin S$, then $x_{\mathfrak{p}}$ is cyclic
and if $x_{\mathfrak{p}}$ is ramified then $\mathfrak{p}$ splits
completely in $\Omega$.
\end{thm}
\begin{claim}
\label{The-element x_p}The element \textit{$x\in\textrm{H}^{1}(k,A)$},
which is obtained in Theorem \ref{theorem1.neu}, is ramified at no
more then $3\cdot rank(A)$ primes which are not in $S$.
\end{claim}
\begin{proof}
The element is given as the sum of two other elements: $x=\zeta+z\in\textrm{H}^{1}(k,A)$.
Let $V=S\cup\{\mathfrak{p}\mid z_{\mathfrak{p}}\ is\ ramified\}$.
In the proof of Theorem \ref{theorem1.neu} it is shown that if $\mathfrak{p}\in V-S$
then $x_{\mathfrak{p}}=0$, namely $x$ is unramified at $\mathfrak{p}$.
If $x_{\mathfrak{p}}$ is ramified for $\mathfrak{p}\notin V$ then
since $z_{\mathfrak{p}}$ is unramified we must have that $\zeta_{\mathfrak{p}}$
is ramified. The element $\zeta$ is given as the image of an element
$\overline{\zeta}\in\textrm{H}^{1}(K,A)$ which is the element obtained
in proposition \ref{prop:main.lemma}. Namely, $\zeta$ can ramify
at no more then $3\cdot rank(A)$ primes outside of $S$, and thus
$x$ can be ramified at no more then $3\cdot rank(A)$ primes which
are not in $S$.
\end{proof}

\section{{\large{}Upper bound for all odd order groups.}}

Let $G$ and $G'$ be profinite groups with a given homomorphisms
$f,f'$ into a finite group $\Gamma$. We consider all $\psi\in\textrm{Hom}(G,G')$
for which the following diagram:
\begin{align*}
 & \begin{array}{ccc}
G & \overset{\psi}{\rightarrow} & G'\\
\downarrow & \swarrow\\
\Gamma
\end{array}\\
\end{align*}

commutes. Let $\psi_{1},\psi_{2}$ be two such elements, they are
called equivalent if there exists an element $a\in ker(f')$ such
that:
\begin{eqnarray*}
 & \forall x\in G\ \psi_{2}(x)=a\psi_{1}(x)a^{-1}
\end{eqnarray*}
namely if $\psi_{1}$ and $\psi_{2}$ are conjugates by some element
of $ker(f')$. The equivalence classes are denoted by $[\psi]$ and
$\mathcal{HOM}{}_{\Gamma}(G,G')$ denotes the set of all equivalence
classes. Moreover, the set of surjective representatives is denoted
by $\mathcal{HOM}{}_{\Gamma}(G,G')_{sur}$. Let $[\psi]\in\mathcal{HOM}{}_{\Gamma}(G_{k},G)$
and let us consider the canonical restriction map:
\begin{eqnarray*}
\mathcal{HOM}{}_{\Gamma}(G_{k},G) & \rightarrow\prod_{\mathfrak{p}} & \mathcal{HOM}{}_{\Gamma}(G_{k_{p}},G)\\
\end{eqnarray*}
An element $[\psi]\in\mathcal{HOM}{}_{\Gamma}(G_{k},G)$ is called
unramified at $\mathfrak{p}$ if the inertia group $I_{k_{\mathfrak{p}}}$
of $G_{k_{p}}$ is contained in the kernel of $\psi_{\mathfrak{p}}$,
where $[\psi_{\mathfrak{p}}]\in\mathcal{HOM}{}_{\Gamma}(G_{k_{p}},G)$
is the corresponding element.

Let $k$ be an algebraic number field. Assume that $\varphi$ is an
homomorphism of $G_{k}$ into the finite group $\Gamma$, and denote
by $K$ the kernel of $\varphi$. Let us look at the following embedding
problem:
\begin{eqnarray*}
 &  & \mathfrak{G}\\
 & \swarrow & \begin{array}{cc}
\downarrow & \varphi\end{array}\\
1\rightarrow A\rightarrow G & \rightarrow & \ \Gamma\rightarrow1
\end{eqnarray*}
where $A$ is an elementary abelian group. Let $E\rightarrow G$ be
a surjective homomorphism with a solvable kernel of exponent $e$
and let $n$ be a multiple of $ep$.
\begin{lem}
\label{lemma6 neu}\cite{key-12} Let $S$ be an arbitrary set of
primes of $k$ and assume\footnote{m($K$) denotes the number of roots of unity in $K$.}
that $(n,m(K))=1$. If $\underset{\mathfrak{p}}{\prod}\mathcal{HOM}{}_{\Gamma}(G_{k_{p}},G)\neq\phi$,
then there exists an element $\psi\in\mathcal{HOM}{}_{\Gamma}(G_{k},G)_{sur}$
with the following properties:\\
1) $\psi$ induces given elements $\psi_{\mathfrak{p}}\in\mathcal{HOM}{}_{\Gamma}(G_{k_{p}},G)$
at primes $\mathfrak{p}\in S$.\\
2) If $\mathfrak{p}$ is a prime of $k$ which is not in $S$ and
is unramified in $K\mid k$, then $\mathcal{HOM}{}_{G}(G_{k_{p}},E)\neq\phi$.
\\
3) For the field $N$ defined by $\psi:G_{k}\rightarrow G$ we have
$(n,m(N))=1$.
\end{lem}
\begin{rem}
The group $E$ in the above lemma is needed in order to keep solving
the embedding problems.
\end{rem}
\begin{lem}
\label{element obt from lemma 6}Under the above notation, the element
obtained in Lemma \ref{lemma6 neu} is ramified outside of $S$ at
no more then $3\cdot rank(A)$ primes.
\end{lem}
\begin{proof}
By assumption $\prod_{\mathfrak{p}}\mathcal{HOM}{}_{\Gamma}(G_{k_{p}},G)\neq\phi$
then by\textit{ }\cite[Lemma 4]{key-12}, the set $\mathcal{HOM}{}_{\Gamma}(G_{k},G)\neq\phi$
so we can start with $[\psi_{0}]\in\mathcal{HOM}{}_{\Gamma}(G_{k},G)$.
Denote by $N_{0}\mid k$ the field defined by $\psi_{0}$ and let
$\Omega=N_{0}(\zeta_{n}),$ where $\zeta_{n}$ denotes a primitive
$n$-th root of unity. Let $\mathfrak{p}_{1},...,\mathfrak{p}_{r}$
be primes outside of $S$ for which $\psi_{0}$ is ramified, and denote
$S^{*}=S\cup\{\mathfrak{p}_{1},...,\mathfrak{p}_{r}\}$. Thus, the
primes $\mathfrak{p}\in S^{*}-S$ are unramified in $K\mid k$, namely
the homomorphism $\varphi_{\mathfrak{p}}$ is unramified, and therefore
can be lifted\footnote{\cite[Lemma 5]{key-12}} to an unramified $\Gamma-$homomorphism
$\tilde{\psi}_{\mathfrak{p}}$. For each $\mathfrak{p}\in S^{*}$
let $y_{\mathfrak{p}}\in\textrm{H}^{1}(G_{k_{p}},\mathbb{Z}/p)$ be
the cohomology class which sends $[(\psi_{0})_{\mathfrak{p}}]$ into
$[\tilde{\psi}_{\mathfrak{p}}]$, namely:
\begin{eqnarray*}
[(\psi_{0})_{\mathfrak{p}}]^{y_{\mathfrak{p}}} & = & [\psi_{\mathfrak{p}}]\\
\end{eqnarray*}
where $\tilde{\psi}_{\mathfrak{p}}$ are the elements which are given
in advance. We now apply Theorem \ref{theorem1.neu} and obtain an
element $x\in\textrm{H}^{1}(\mathfrak{G},A)$ such that:

1) $x_{\mathfrak{p}}=y_{\mathfrak{p}}$ for $\mathfrak{p}\in S^{*}$.

2) if $\mathfrak{p}\notin S^{*}$ then $x_{\mathfrak{p}}$ is cyclic
and if $x_{\mathfrak{p}}$ is ramified then $\mathfrak{p}$ splits
completely in $\Omega$.

The element $x$ changes the solution of the embedding problem into
a solution which is ramified outside $S$ at no more then $3\cdot rank(A)$
primes. Namely, $[\psi]$ which is given by: 
\begin{eqnarray*}
[\psi] & = & [\psi_{0}]^{x}\\
\end{eqnarray*}
 has the desired property. Indeed, if $\mathfrak{p}\in S^{*}-S$ then
$[\psi_{\mathfrak{p}}]=[\tilde{\psi}_{\mathfrak{p}}]$ which we know
is unramified. Let $\mathfrak{p}\notin S^{*}$ and assume that $[\psi_{\mathfrak{p}}]$
is ramified. However, $[\psi_{\mathfrak{p}}]=[\psi_{0}]^{x_{\mathfrak{p}}}$
and we know that $[(\psi_{o})_{\mathfrak{p}}]$ is unramified outside
$S^{*}$ so we must have that the cohomology class $x_{\mathfrak{p}}$
is ramified. The lemma now follows from claim \ref{The-element x_p}.
\end{proof}
Let us recall that a $chief\ series$ is a maximal normal series of
a group. Namely, a chief series of a group $G$ is a finite collection
of normal subgroups $N_{i}\subseteq G$: 
\begin{eqnarray*}
 & G=N_{0}\supseteqq N_{1}\supseteqq\cdots\supseteqq N_{t}=\{1\}
\end{eqnarray*}
such that each quotient group $N_{i-1}/N_{i}$ is a minimal normal
subgroup of $G/N_{i}$.
\begin{rem}
\label{chief factors}Note that the chief factors of every chief series
are elementary abelian $l$-group, where $l$ is a prime number. Indeed,
for a solvable group $G$, let $\{G_{i}\}_{i=1}^{m}$ denote the chief
series. Then $G_{m-1}$ is a minimal normal subgroup of $G$ and thus
is elementary abelian $l$-group. The remark now follows by induction
and the fact that $\{G_{i}/G_{m-1}\}_{i=1}^{m-1}$ form a chief series
of the solvable group $G/G_{m-1}$. Lets denote by $p_{i}$ the prime
associated to the group $N_{i-1}/N_{i}$. Note that:
\begin{eqnarray*}
\sum_{i=1}^{t}rank(N_{i-1}/N_{i}) & = & \sum_{i=1}^{t}log_{p_{i}}(|N_{i-1}/N_{i}|)\leq\sum_{i=1}^{t}ln(|N_{i-1}/N_{i}|)=ln(\prod_{i=1}^{t}|N_{i-1}/N_{i}|)=ln(|G|)\\
\end{eqnarray*}
\end{rem}
\begin{thm}
\label{thm:main.neukirch}Let $G$ be an odd order group, then:
\begin{eqnarray*}
 & ram^{t}(G)\leq3\cdot ln(|G|)\\
\end{eqnarray*}
\end{thm}
\begin{proof}
Let us take a chief series of $G$: 
\begin{eqnarray*}
 & G=N_{0}\supseteqq N_{1}\supseteqq\cdots\supseteqq N_{t}=\{1\}
\end{eqnarray*}
 and let $G_{i}=G/N_{i}$. The kernel of the natural map $\pi_{i}:G_{i}\rightarrow G_{i-1}$
will be denoted by $A_{i}=N_{i-1}/N_{i}$. From remark \ref{chief factors}
we know that each chief factor $A_{i}$ is elementary abelian.

The following diagram describes our embedding problems:

\begin{eqnarray*}
 &  & G_{k}\\
 & \begin{array}{cc}
\psi_{i} & \iddots\\
\iddots
\end{array} & \begin{array}{cc}
\downarrow & \psi_{i-1}\end{array}\\
1\rightarrow A_{i}\rightarrow G_{i} & \overset{\pi_{i}}{\rightarrow} & \ G_{i-1}\rightarrow1
\end{eqnarray*}

We will review the construction of an element $[\psi_{i}]\in\mathcal{HOM}{}_{\{e\}}(G_{k},G)$
and count the number of ramified primes in the corresponding field.
The construction is by induction on the chief series length, namely,
it is shown that for each $i=1,...,t$ there exists an epimorphism
$\psi_{i}:G_{k}\rightarrow G_{i}$ with the following properties:

(a) $\psi_{i-1}=\pi_{i}\circ\psi_{i},\ i=1,2,...,t.$

(b) $[\psi_{i}\mid G_{k_{p}}]=[\psi_{\mathfrak{p}}]$ in $\mathcal{HOM}{}_{\{e\}}(G_{k_{p}},G_{i})$
for $\mathfrak{p}\in S$.

(c) For the field $K_{i}$, defined by $\psi_{i},$ we have $(|G|,m(K_{i}))=1$.

(d) $\prod_{\mathfrak{p}}\mathcal{HOM}{}_{G_{i}}(G_{k_{p}},G)\neq0$.

These properties are essential in order to keep solving the embedding
problems. For $i=0$ we simply take $\psi_{0}$ to be the trivial
map. Let us assume that the homomorphism $\psi_{i-1}$ is already
constructed. Let $S_{i-1}$ be the set of all primes of $\mathbb{Q}$
which ramify in the field $K_{i-1}$ which corresponds to $\psi_{i-1}$.
In \cite{key-12} it is shown that we can then apply Lemma \ref{lemma6 neu}
to this situation by replacing: $S$ by $S_{i-1}$, $E\rightarrow G$
by $G\rightarrow G_{i}$ and $\varphi$ by $\psi_{i-1}$. We then
obtain an element $\psi_{i}$, which is a representative of $[\psi_{i}]\in\mathcal{HOM}{}_{G_{i-1}}(G_{k},G_{i})_{sur}$,
with the following properties:

(i) $[\psi_{i}\mid G_{k_{p}}]=[(\psi_{i})_{\mathfrak{p}}]$ in $\mathcal{HOM}{}_{G_{i-1}}(G_{k_{p}},G_{i})$
for $\mathfrak{p}\in S_{i-1}$.

(ii) for $\mathfrak{p}\notin S_{i-1}$ we have $\mathcal{HOM}{}_{G_{i}}(G_{k_{p}},G)\neq0.$

(iii) For the field $K_{i}$, defined by $\psi_{i},$ we have $(|G|,m(K_{i}))=1$.

It is shown that $\psi_{i}$ satisfy (a)-(d). From Lemma \ref{element obt from lemma 6}
we deduce that the element $\psi_{i}$ is ramified outside of $S_{i-1}$
at no more then $3\cdot rank(A_{i})$ primes. Since $S_{i-1}$ is
defined to be the set of ramified primes in $K_{i-1}$ we deduce that
$K_{i}$ has at most $3\cdot rank(A_{i})$ more ramified primes over
$\mathbb{Q}$ than $K_{i-1}$. 

We deduce that the number of ramified primes in the desired field
is at most the sum of the ``new'' ramified primes in each step of
the induction, $\sum_{i=1}^{t}3\cdot rank(A_{i})$. According to Remark
\ref{chief factors} we have:
\begin{eqnarray*}
\sum_{i=1}^{t}rank(A_{i}) & \leq & ln(|G|)\\
\end{eqnarray*}

and thus:
\begin{eqnarray*}
 & ram^{t}(G)\leq\sum_{i=1}^{t}3\cdot rank(A_{i})\leq3\cdot ln(|G|)\\
\end{eqnarray*}
\end{proof}
$ $


\newpage
\begin{thebibliography}{Ki-Ne-So.10}
\bibitem[Bo-Ma.10]{key-8} N. Boston, N. Markin. \textit{The fewest
primes ramified in a G-extension of $\mathbb{Q}$}, Annales des Sciences
math\'{ }ematiques du Qu\'{ }ebec. 33 no 2, 145\textendash 154 (2009).

\bibitem[De.95]{key-2} R. Dentzer, \textit{On geometric embedding
problems and semiabelian groups}. Manuscripta Math. 86 (1995), 199-216.

\bibitem[Ge-Ja.98]{key-6} W.-D. Geyer and M. Jarden, \textit{Bounded
realization of l-groups over global fields. The method of Scholz and
Reichardt} , Nagoya Math. J. 150 (1998), 13-62.

\bibitem[Ja.79]{key-3} N. Jacobson,\textit{ Lie Algebras, 2nd ed},
Dover, New York, (1979).

\bibitem[Ki-So.10]{key-14} H. Kisilevsky and J. Sonn, \textit{On
the minimal ramification problem for $l$-groups}, Compositio Math.
146 (2010) 599-606.

\bibitem[Ki-Ne-So.10]{key-9} H. Kisilevsky, D. Neftin, and J. Sonn,
\textit{On the minimal ramification problem for semiabelian groups},
Algebra and Number Theory 4 no.8 (2010), 1077-1090.

\bibitem[Ne.11]{key-17}D. Neftin, \textit{On semiabelian p-groups},
J. Algebra 344 (2011), 60-69.

\bibitem[Ne.79]{key-12} J. Neukirch,\textit{ On Solvable Number Fields},
Inventiones Mathematicae 53 (1979), 135\textendash 164.

\bibitem[Pl.04]{key-5} B. Plans, \textit{On the minimal number of
ramified primes in some solvable extensions of} $\mathbb{Q}$, Pacific
J. Math. 215 (2004), 381-391.

\bibitem[Re.37]{key-7} H. Reichardt, \textit{Konstruction von Zahlkörpern
mit gegebener Galois gruppe von Primzahlpotenzordnung}, J. Reine Angew.
Math. 177 (1937), 1-5.

\bibitem[Se.92]{key-4} J.-P. Serre,\textit{ Topics in Galois Theory},
Jones and Bartlett, Boston (1992).

\bibitem[Wi.35]{key-1} E. Witt, \textit{Der Existenzsatz fur abelsche
Funktionenkorper}, J. Reine Angew. Math, 173 (1935), 43-51.
\end{thebibliography}
\end{document}